\documentclass[preprint,12pt]{elsarticle}
\usepackage{amssymb}
 \usepackage{amsthm}
 \newtheorem{lemma}{Lemma}[section]
 \newtheorem{theorem}{Theorem}[section]

\journal{Appl. Math. Lett.}
\begin{document}
\begin{frontmatter}

\title{Connectivity of direct products by  $K_2$}
\author{Wei Wang  }
\author{Zhidan Yan \corref{cor1}}
\cortext[cor1]{Corresponding author.
Tel:+86-997-4680821; Fax:+86-997-4682766.\\
\ead{yanzhidan.math@gmail.com}}

\address{College of Information Engineering, Tarim University, Alar 843300, China}

\begin{abstract}
Let $\kappa(G)$ be the connectivity of $G$. The Kronecker product
$G_1\times G_2$ of graphs $G_1$ and $G_2$ has vertex set
$V(G_1\times G_2)=V(G_1)\times V(G_2)$ and edge set $E(G_1\times
G_2)=\{(u_1,v_1)(u_2,v_2):u_1u_2\in E(G_1),v_1v_2\in E(G_2)\}$. In
this paper, we prove that
 $\kappa(G\times K_2)=\textup{min}\{2\kappa(G) , \textup{min}\{|X|+2|Y|\}\}$,  where
the second minimum is taken over all  disjoint sets $X,Y\subseteq
V(G)$ satisfying (1)$G-(X\cup Y)$ has a bipartite  component $C$,
and (2) $G[V(C)\cup \{x\}]$ is also bipartite for each $x\in X$.

\end{abstract}

\begin{keyword}
connectivity \sep Kronecker product \sep separating set \MSC 05C40
\end{keyword}

\end{frontmatter}

\section{Introduction}
\label{intro} Throughout this paper, a graph $G$ always means a
finite undirected graph without loops or multiple edges. It is well
known that a network is
 often modeled as a graph and the
classical measure of the reliability is the connectivity
 and the edge connectivity.

  The connectivity of a simple graph $G = (V(G), E(G))$ is the number,
  denoted by $\kappa(G)$,   equal to the fewest number of vertices whose
   removal from $G$ results in a disconnected or trivial graph. A set $S\subseteq V(G)$
   is a separating set of $G$, if either $G-S$ is   disconnected or has
   only one vertex. Let $G_1$ and $G_2$ be two graphs, the Kronecker
   product $G_1\times G_2$ is the graph defined on the Cartesian
   product of vertex sets of the factors, with two vertices
   $(u_1,v_1)$ and $(u_2,v_2)$ adjacent if and only if $u_1u_2\in E(G_1)$
   and $v_1v_2\in E(G_2)$. This product is one of the  four standard graph products \cite{imrich2000} and is known under many different
   names, for instance as the direct product, the cross product and conjunction.
   Besides the famous Hedetniemi's conjecture on chromatic number
   of Kronecker product of two graphs, many different properties of Kronecker product have been investigated,
   and this product has several applications.

   Recently, Bre\v{s}ar and \v{S}pacapan \cite{bresar2008} obtained an
   upper bound and a low bound on the edge connectivity of Kronecker products with some exceptions; they
   also obtained several upper bounds on the vertex connectivity of
   the Kronecker product of graphs. Mamut and Vumar \cite{mamut2008} determined the
   connectivity of  Kronecker product of two complete graphs. Guji and Vumar \cite{guji2009} obtained the
   connectivity of Kronecker product of a bipartite graph and a
   complete graph. These two results are generalized in \cite{wang2010},
   where the author proved a formula for the
   connectivity of Kronecker product of an arbitrary graph and a
   complete graph of order $\ge 3$, which was conjectured in
   \cite{guji2009}. We mention that a different proof of the same result can be found in \cite{wangappear}. For the left case $G\times K_2$,
   Yang \cite{yang2007} determined an explicit formula for its edge
   connectivity. Bottreau and M\'{e}tivier \cite{bottreau1998} derivied a criterion for
   the existence of a cut vertex  of $G\times K_2$, see also
   \cite{hafner1978}.  In this paper, based on a similar argument of  \cite{wangappear}, we determine a formula for $\kappa(G\times
   K_2)$. For more study on the connectivity of  Kronecker product
   graphs, we refer to \cite{guo2010,ou2011}.

   Let $X,Y\subseteq V(G)$ be two disjoint sets with $V(G)-(X \cup   Y)\ne \emptyset$.
   We shall call $(X,Y)$  a \emph{b-pair} of $G$
   if it satisfies:\\
   (1). $G-(X\cup Y)$ has a bipartite component $C$, and\\
   (2). $G[V(C)\cup\{x\}]$ is also bipartite for each $x\in X$.

   Denote $b(G)=\textup{min}\{|X|+2|Y|:(X,Y)\ \textit{is a b-pair
   of}\ G\}$. Our main result is the following
   \begin{theorem}\label{main}
       $\kappa(G\times K_2)=\textup{min}\{2\kappa(G),b(G)\}$.
   \end{theorem}

We end this section by giving some useful properties of the new
defined graph invariant $b(G)$. Let $v\in V(G)$, we use $N(v)$,
$d(v)$ and $\delta(G)$ to denote the neighbor set of $v$, the degree
of $v$, and the minimum degree of $G$, respectively.

\begin{lemma}\label{newinv}
Let $m=|G|\ge2$ and $u$ be any vertex of $G$. Then\\
\textup{(1)}.$b(G)=0$, if $G$ is bipartite.\\
\textup{(2)}.$b(G)\le \delta(G)$.\\
\textup{(3)}.$b(G)\le b(G-u)+2$.

\end{lemma}
\begin{proof}
Part (1) is clear since $(\emptyset,\emptyset)$ is a b-pair of any
bipartite graph by the definition of b-pairs. For each $v\in V(G)$,
$(N(v),\emptyset)$ is a b-pair of $G$(Take the isolated vertex $v$
in $G-N(v)$ as the bipartite component $C$). Therefore $b(G)\le
d(v)$ and part (2) is verified. Similarly, Let $(X',Y')$ be any
b-pair of $G-u$. It is straightforward to show that
$(X',Y'\cup\{u\})$ is a b-pair of $G$. Therefore, $b(G)\le
|X'|+2|Y'|+2$ and part (3) is verified.
\end{proof}

\section{Proof of the main result}
 We first recall some basic results on the connectivity of Kronecker
 product of graphs \cite{weichsel1962}, see also \cite{imrich2000}.

 \begin{lemma}\label{basic}
 The Kronecker product of two nontrivial graphs is connected if and
 only if both factors are connected and at least one factor is
 nonbipartite. In particular, $G\times K_2$ is connected if and only
 if $G$ is a connected nonbipartite graph.
 \end{lemma}

 \begin{lemma}\label{bip2}
 Let $G$ be a connected bipartite graph with bipartition $(P,Q)$ and $V(K_2)=\{a,b\}$. Then
  $G\times K_2$ has exactly two connected components isomorphic to
  $G$, with bipartitions $(P\times \{a\},Q\times \{b\})$ and $(P\times \{b\},Q\times \{a\})$, respectively.
 \end{lemma}

From Lemma \ref{newinv}(1) and Lemma \ref{basic}, we only need to
proof Theorem \ref{main} for connected and nonbipartite graphs. For
each $u\in V(G)$, set $S_u=\{u\}\times V(K_2)=\{(u,a),(u,b)\}$.
Let $S\subseteq V(G\times K_2)$ satisfy the following two assumptions: \\
\textbf{Assumption 1}. $|S|<\mbox{min}\{2\kappa(G),b(G)\}$, and\\
\textbf{Assumption 2}. $S'_u:=S_u-S\ne\emptyset$ for each $u\in
V(G)$.

Let $G^*$ be the graph whose vertices are the classes $S'_u$ for all
$u\in V(G)$ and in which two different vertices $S'_u$ and $S'_v$
are adjacent if $G\times K_2-S$ contains an $(S'_u-S'_v)$ edge, that
is, an edge with one end in $S'_u$ and the other one in $S'_v$.

Under the two assumptions on $S\subseteq V(G\times K_2)$ , the
connectedness of  $G\times K_2-S$ is verified by the following two
lemmas.

 \begin{lemma}\label{qut}
 If $G$ is a connected nonbipartite graph then $G^*$ is connected.
 \end{lemma}

 \begin{proof}
 Suppose to the contrary that $G^*$ is disconnected. Then the
 vertices of $G^*$ can be partitioned into two nonempty parts, $U^*$
 and $V^*$, such that there are no $(U^*-V^*)$ edges. Let
 $U=\{u\in V(G):S'_u\in U^*\}, V=\{v\in V(G):S'_v\in V^*\}$ and $Z$ be the
 collection of ends of all $(U-V)$ edges. Let $Z^*=\{S'_u:u\in
 V(G),|S'_u|=1\}$. For any $u\in Z$, there exists an edge $uv\in
 E(U,V)$. It follows that both $S'_u$ and $S'_v$ contain exactly one
 element since otherwise  $G\times K_2-S$ contains an $(S'_u-S'_v)$
 edge, i.e., $S'_uS'_v\in E(G^*)$, which is contrary to the fact
 that there are no $(U^*-V^*)$ edges. Therefore, $S'_u\in Z^*$ and we
 have $|Z|\le|Z^*|$ by the arbitrariness of $u$ in $Z$.

  \textbf{case 1}: Either $U\subseteq Z$ or $V\subseteq Z$. We may
  assume $U\subseteq Z$. Let $u$ be any vertex in $U$, then  $d(u)\le|Z|-1$,
  and hence $\delta(G)\le|Z|-1$. Therefore, by  Lemma \ref{newinv}(2), we have
   $|S|=|Z^*|\ge|Z|>\delta(G)\ge b(G)$, a contradiction.

 \textbf{case 2}: $U\nsubseteq Z$ and  $V\nsubseteq Z$. Either of
 $U\cap Z$ and $V\cap Z$ is a separating set of $G$. Therefore,
 $\kappa(G)\le\mbox{min}\{|U\cap Z|,|V\cap Z|\}\le|Z|/2$. Similarly,
 we have $|S|=|Z^*|\ge|Z|\ge2\kappa(G)$, again a contradiction.
\end{proof}

\begin{lemma}\label{com} Any vertex $S'_w$ of $G^*$, as a subset of $V(G\times
K_2-S)$ is contained in the vertex set of some component of $G\times
K_2-S$.
\end{lemma}
\begin{proof}
If $|S'_w|=1$, then the assertion holds trivially. Now assume
$|S'_w|=2$. Let $U=\{u\in V(G):|S'_u|=2\}$, $V=\{v\in
V(G):|S'_v|=1\}$ be the partitions of $V(G)$ and $C$ the component
of $G-V$ containing $w\in U$.

Since $|V|=|S|<b(G)$ by Assumption 1, it follows that
$(V,\emptyset)$ is not a b-pair of $G$. Note $S'_w\subseteq
V(C\times K_2)$. We may assume that the component $C$ containing $w$
is bipartite since otherwise $C\times K_2$ is connected by Lemma
\ref{basic} and hence the result follows. Therefore, by the
definition of b-pairs, there exists a vertex $v\in V$ such that
$G[V(C)\cup\{v\}]$ is nonbipartite.

 Let $(P,Q)$ be the bipartition of $C$ and
$V(K_2)=\{a,b\}$. Then, by Lemma \ref{bip2}, $C\times K_2$ has
exactly two components $C_1$ and $C_2$ isomorphic to $C$, with
bipartition $(P\times \{a\},Q\times \{b\})$ and $(P\times
\{b\},Q\times \{a\})$, respectively. The nonbipartiteness of
$G[V(C)\cup\{v\}]$ implies that $v$ is a common neighbor of $P$ and
$Q$. By symmetry, we may assume $S'_v=\{(v,a)\}$. It is easy to see
that the subgraph induced by $V(C\times K_2)\cup S'_v$ is connected
since $(v,a)$ is a common neighbor of  $C_1$ and $C_2$.
\end{proof}

\textbf{Proof of Theorem \ref{main}} We apply induction on
$m=|V(G)|$.  It trivially holds when $m=1$. We therefore assume
$m\ge 2$ and the result holds for all graphs of order $m-1$.

Let $S_0$ be a minimum separating set of $G$ and $S=S_0\times
V(K_2)=\{(u,a),(u,b):u\in S_0\}$. Then $G\times K_2-S\cong
(G-S_0)\times K_2$ is disconnected by Lemma \ref{basic}. Therefore,
$\kappa(G\times K_2)\le2\kappa(G)$.

Let $(X,Y)$ be a b-pair of $G$ with $|X|+2|Y|=b(G)$.  Let $C$ be a
bipartite component of $G-(X\cup Y)$ with bipartition $(P,Q)$ such
that $G[V(C)\cup\{x\}]$ is also bipartite for each $x\in X$. Let
$C_1$ and $C_2$ be the two components of $C\times K_2$ with
bipartition $(P\times \{a\},Q\times \{b\})$ and $(P\times
\{b\},Q\times \{a\})$, respectively. Define an
injection $\varphi:X\rightarrow V(G\times K_2)$ as follows:\\
\[
\varphi(x) = \left\{ {\begin{array}{c@{{},{}}l}
   (x,b)& \textit{if x has a neighbor in P,}  \\
   (x,a)& otherwise.
\end{array}} \right.
\]
Let $S'=\varphi(X)$ and $S''=\{(u,a),(u,b):u\in Y\}$. Then $S'\cup
S''$ is a separating set since $C_1$ is a component of $G\times
K_2-(S'\cup S'')$, which implies $\kappa(G\times K_2)\le|S'\cup S''|
=|X|+2|Y|=b(G)$.

To show the other equality, we may assume $G$ is a connected
nonbipartite graph. Let $S\subseteq V(G\times K_2)$ satisfy
Assumption 1, i.e., $|S|<\mbox{min}\{2\kappa(G),b(G)\}$.

\textbf{case 1:} $S$ satisfies Assumption 2. It follows by Lemma
\ref{qut} and Lemma \ref{com} that $G\times K_2-S$ is connected.

\textbf{case 2:} $S$ does not satisfy Assumption 2, i.e., there
exists a vertex $u\in V(G)$ such that $S_u=\{(u,a),(u,b)\}\subseteq
S$. Therefore,
\begin{eqnarray*}
 |S-S_u|&=&|S|-2\\
       &<&\mbox{min}\{2\kappa(G),b(G)\}-2\\
       &=&\mbox{min}\{2(\kappa(G)-1),b(G)-2\}\\
       &\le&\mbox{min}\{2\kappa(G-u),b(G-u)\},
\end{eqnarray*}
where the last inequality above follows from Lemma \ref{newinv}(3).

By the induction assumption,
\begin{displaymath}
  \kappa((G-u)\times K_2)=\mbox{min}\{2\kappa(G-u),b(G-u)\}.
\end{displaymath}

Hence, $(G-u)\times K_2-(S-S_u)$ is connected. It follows by
isomorphism that  $G\times K_2-S$ is connected.

Either of the two cases implies that $(G\times K_2-S)$ is connected.
Thus, $\kappa(G\times K_2)\ge\mbox{min}\{2\kappa(G),b(G)\}$.

The proof of the theorem is complete by induction. \qed


\begin{thebibliography}{99}

 \bibitem{imrich2000} W. Imrich, S. Klav\v{z}ar, Product Graphs: Structure and Recognition, Wiley, 2000.


 \bibitem{bresar2008} B. Bre\v{s}ar, S. \v{S}pacapan,  On the connectivity of the direct product of graphs,
Australas. J. Combin. 41 (2008) 45-56.

  \bibitem{mamut2008} A. Mamut, E. Vumar, Vertex vulnerability parameters of Kronecker product of complete graphs,
   Inform. Process. Lett. 106 (2008)  258-262.

  \bibitem{guji2009}R. Guji, E. Vumar,  A note on the connectivity of Kronecker products of
graphs, Appl. Math. Lett. 22 (2009) 1360-1363.

\bibitem{wang2010}Y. Wang,  The problem of partitioning of graphs into connected
subgraphs and the connectivity of Kronecker product of graphs, M.D.
Thesis, University of Xinjiang, 2010.

\bibitem{wangappear} W. Wang, N.N.Xue, Connectivity of direct products of graphs,
arXiv:1102.5180v1.

\bibitem{yang2007} C. Yang, Connectivity and fault-diameter of product graphs, Ph.D.
Thesis, University of science and technology of China, 2007.

\bibitem{bottreau1998} A. Bottreau, Y. M\'{e}tivier, Some remarks on the
Kronecker product of graphs, Inform. Process. Lett. 68 (1998) 55-61.

 \bibitem{hafner1978} P. Hafner, F. Harary, Cutpoints in the
 conjunction of two graphs, Arch. Math. 31 (1978) 177-181.


\bibitem{guo2010}  L. T. Guo,   C. F. Chen,  X. F. Guo, Super connectivity
of Kronecker products of graphs, Inform. Process. Lett. 110 (2010)
659-661.



\bibitem{ou2011} J. P. Ou, On optimizing edge connectivity of product graphs,
Discrete Math. 311 (2011) 478-492.





 \bibitem{weichsel1962} P. M. Weichsel, The Kronecker product of
 graphs, Proc. Amer. Math. Soc. 13 (1962) 47-52.











 \end{thebibliography}
\end{document}